\documentclass[11pt, a4paper]{amsart}

\usepackage[utf8]{inputenc} 
\usepackage[T1]{fontenc}      % Police contenant les caractères français
\usepackage{geometry}         % Définir les marges
\usepackage[english]{babel}  % Placez ici une liste de langues, la
\usepackage{amsmath}                            
\usepackage{amsfonts}
\usepackage{amsthm}                             
                             
\usepackage{amscd}

\usepackage{hyperref}

   \usepackage{amssymb}
   \usepackage{verbatim}
   \newif\ifpdf
   \ifpdf
     \usepackage[pdftex]{graphicx}
     \usepackage[pdftex]{hyperref}
   \el

     \usepackage{graphicx}
   \fi
   
\usepackage{dsfont}\let\mathbb\mathds
\usepackage[all]{xy}

\newcommand{\N}{\mathbb{N}} 
 
\newcommand{\C}{\mathbb{C}}

\newcommand{\teich}{\mathrm{Teich}}
\newcommand{\belf}{\mathrm{bel}(f)}
\newcommand{\Belf}{\mathrm{Bel}(f)}

\newtheorem{theo}{Theorem}  
\newtheorem{prop}{Proposition}   
      
\newtheorem{defi}{Definition} 
\newtheorem{lem}{Lemma} 
\newtheorem{rem}{Remark}

\newtheorem*{theoA}{Theorem A}
\newtheorem*{theoB}{Theorem B}

\newcommand{\dbar}{\overline{\partial}}
\newcommand{\rs}{\mathbb{P}^1}
\newcommand{\res}{\mathrm{Res}}
\newcommand{\ratd}{\mathrm{Rat}_d}

\newcommand{\limn}{\lim_{n \rightarrow \infty}}
\newcommand{\id}{\mathrm{Id}}
\newcommand{\critf}{\mathrm{Crit}(f)}

\newcommand{\mcal}{\mathcal{M}}
\newcommand{\ddf}{\Gamma(T\s)^*}

\newcommand{\qc}{\mathrm{QC}}
\newcommand{\defab}{\mathrm{Def}_A^B}
\newcommand{\s}{X}
\newcommand{\fatou}{\mathcal{F}}
\newcommand{\julia}{\mathcal{J}}
\newcommand{\modspace}{\mathrm{rat}_d}

\newcommand{\card}{\mathrm{card\,}}
\newcommand{\acal}{\mathcal{A}}

\newcommand{\ucal}{\mathcal{U}}

\title{Summability condition and rigidity for finite type maps}

\author{Matthieu Astorg}

\email{mastorg@umich.edu}
\address{University of Michigan, 530 Church St, 48109 Ann Arbor, MI}

\begin{document}
	
\begin{abstract}
	We give a bound on the dimension of the Teichmüller space of a finite type map in terms of the 
	number of summable singular values and finite singular orbits. This generalizes rigidity 
	results due to Makienko, Dominguez and Sienra, and Urbanski and Zdunik.
%	We extend a series of results due to Makienko, Dominguez and Sienra, and Urbanski and Zdunik, concerning the rigidity of 
%	some holomorphic dynamical systems with summable singular values to the more general setting of 
%	finite type maps. 
	We also recover a shorter proof of a transversality theorem due to Levin.
	Our methods are based on the deformation theory introduced by Epstein.
\end{abstract}

\maketitle

\section{Introduction and statement of the main theorems}

Finite type maps are a class of analytic maps $f: W \to X$ of complex 1-manifolds introduced and studied
by A. Epstein in \cite{epstein1993towers}. More precisely:

\begin{defi}
	An analytic map $f: W \to X$ of complex $1$-manifolds is of finite type if 
	\begin{itemize}
		\item $f$ is nowhere locally constant,
		\item $f$ has no isolated removable singularities,
		\item $X$ is a compact Riemann surface, and
		\item the singular set $S(f)$ is finite.
	\end{itemize}
\end{defi}

Here, the singular set $S(f)$ means the smallest set $S\subset X$ such that $f: W \backslash f^{-1}(S) \to X \backslash S$ is a covering map.

This class includes notably rational self-maps of $\rs$ and more generally ramified covers between compact Riemann surfaces,
as well as entire functions of the complex plane with a finite singular set, such as
the exponential family $f_\lambda(z)=\lambda e^z$. 
It also contains the so-called horn maps appearing in the theory of parabolic implosion,
as was proved by Buff, Ecalle and Epstein (see \cite{buff2013limits} and also \cite{epstein1993towers}).
When $W \subset X$, one can study the dynamics
of the map $f$, that is the orbits $(z, f(z), f^2(z)\ldots )$, for as long as $f^n(z) \in W$.
If $z \in W$ is such that for all $n \in \N$, $f^n(z)  \in W$, then we say that $z$ is non-escaping.
We may define the Fatou set $\fatou(f)$ of $f$ as the set of points $z \in W$ such that there exists a neighborhood $U$
of $z$ in $W$ such that either all points in $U$ escape, or the family $\{f_{|U}^n: W \to X, n \in \N\}$ is 
well-defined and normal. The Julia set is $\julia(f)=X-\fatou(f)$.

Epstein proved in his thesis \cite{epstein1993towers} several key results about the dynamics of these maps: they do not possess wandering domains,
their Julia set is never empty, and as for rational maps, we have a classification theorem for periodic Fatou 
components.
He also constructed so-called deformation spaces $\defab(f)$, which are defined abstractly 
through Teichmüller theory but may be
thought of as natural parameter spaces for $f$. 
These deformation spaces are finite-dimensional complex manifolds,
and one can describe their cotangent bundle in terms of quadratic differentials.

Moreover, if $f: W \to X$ is a finite type map, one can define the Teichmüller space of $f$,
denoted by $\teich(f)$.
This notion was first introduced by McMullen and Sullivan (\cite{mcmullen1998quasiconformal})
in the context of algebraic correspondances, and was specifically described in the case of rational maps.
Epstein studied in \cite{epstein1993towers} the case of the Teichmüller spaces of finite type maps.
 Again, this is a finite-dimensional complex manifold, and it comes 
equipped with a natural holomorphic
immersion into the deformation space (see \cite{astorg2014dynamical} for a detailed 
construction of $\teich(f)$ from first principles in the case where $f$ is a rational map, and the proof 
of the fact that $\teich(f)$ immerses into the moduli space).
Roughly speaking, the Teichmüller space of $f$ represents the topological conjugacy class of $f$:
the larger its dimension is, the more parameters of topological deformation exist for $f$.
If $\teich(f)$ is reduced to a point, then $f$ is said to be rigid: this implies that for every
quasiconformal homeomorphism $\phi: X \to X'$ mapping $W$ to $W' \subset X'$, if there is a holomorphic map $g$
making the following diagram commute:
$$
\xymatrix{
W \ar[r]^{f} \ar[d]_{\phi}& X \ar[d]^{\phi} \\ 
W' \ar[r]^{g} & X'  \\ 	
	}
$$
then $X$ must be biholomorphic to $X'$ and $\phi$ must be the composition of a biholomorphism $X \to X'$ and 
of a quasiconformal homeomorphism $X \to X$ commuting with $f$. Thus $g$ is a "trivial" deformation 
of $f$, as it is in fact biholomorphically conjugate to $f$.
On the other hand, if $\teich(f)$ has maximal dimension (that is, equal to the dimension of the moduli space),
then every nearby parameter is quasiconformally conjugate to $f$; we say that $f$ is structurally stable.

The question of describing when a rational map $f: \rs \to \rs$ is structurally stable or rigid
is equivalent to a central conjecture in holomorphic dynamics, which state that with the exception of 
one well-understood family (flexible Lattès maps), no rational map may have an invariant line field supported
on its Julia set.

In the series of papers  \cite{makienko2001remarks}, \cite{makienko2005remarks},
\cite{dominguez2002ruelle}, \cite{makienko2005poincare}, \cite{urbanski2007instability} and 
\cite{makienko2008remarks}, Makienko, Dominguez and Sienra, and Urbanski and Zdunik
proved that some sufficient expansion along 
at least one critical orbit was an obstruction to structural stability, and that expansion along 
\emph{all} critical orbits implied in fact rigidity. This was proved first for rational maps
 in \cite{makienko2001remarks}, under some unnecessary assumptions that were removed in  \cite{makienko2005remarks} (using ideas from \cite{levin2002analytic}), then for 
the exponential family in \cite{makienko2005poincare} and \cite{urbanski2007instability}, and at last for some subset of the class of entire functions with only 
finitely many singular values in \cite{dominguez2002ruelle}. Avila gave a different proof in 
\cite{avila2001infinitesimal} of this result in the case of rational maps.
In particular, all of these maps are of finite type. 
More precisely, this is what we mean by sufficient expansion:

\begin{defi}
	Let $f: W \to X$ be a  finite type map, with $W \subset X$, and let $z \in W$ be a non-escaping point.
	We say that $z$ is \emph{summable} if there is a Hermitian metric on $X$ such that the series
	\begin{equation*}
	\sum_{n \geq 0} \| Df^n(z)\|^{-1} 
	\end{equation*}
	is convergent.
\end{defi}
Note that by compacity of $X$, the choice of the metric does not matter.
This type of condition was first introduced by Tsujii (\cite{tsujii2000simple}) for 
real quadratic polynomials.
Stronger expansivity conditions (the so-called Collet-Eckmann condition) had previously been known 
to imply rigidity (see  \cite{przytycki1999collet}).

We will also need the following definitions before we can state our main results:

\begin{defi}
	A compact set $K \subset \s$ is called a $C$-compact if it satisfies the following property: 
	any continuous function
	on $K$  can be uniformly approximated  by restrictions of 
	functions that are holomorphic  on  a neighborhood of $K$.
\end{defi}

This condition, though not always satisfied by the Julia sets of rational maps, 
is relatively mild; it is in particular always satisfied by Julia sets of polynomials
(see \cite{levin2011perturbations}).

\begin{defi}
	Let $f: W \to X$ be a finite type map, with $W \subset X$. 
	\begin{itemize}
		\item Let $p(f)$ denote the number of singular values with a periodic or preperiodic orbit
		\item Let $s(f)$ denote the number of summable singular values with an infinite forward orbit,
		whose $\omega$-limit sets are $C$-compacts.
	\end{itemize}
\end{defi}

\begin{defi}
	Let $f : W \to X$ be a finite type analytic map. Then we say that $f$ is exceptional if either
	$f$ is an automorphism of $X$, or an endomorphism of a complex torus, or a flexible Lattès example.
\end{defi}

The following is the main result, and generalizes the aforementienned results of 
Avila, Makienko, Dominguez, Sienra, Urbanski and Zdunik:

\begin{theoA}
	Let $f: W \to X$ be a non-exceptional finite type analytic map, with $W \subset X$.
	We have: $$\dim \teich(f) \leq \card S(f) - p(f) - s(f).$$
\end{theoA}

In particular, if at least one singular value is summable with 
an $\omega$-limit set that is a $C$-compact, then $f$ is not structurally stable, 
and if all singular values either are summable with $C$-compacts 
as $\omega$-limit sets, or have finite orbit, then $f$ is rigid and therefore
does not have any invariant line field.

Our second result is a simplified proof of a theorem of Levin (\cite{levin2011perturbations}).
Before we state it, let us introduce some notations.

\begin{defi}
	Let $\ratd$ be the space of degree $d$ rational maps, and let $\modspace$ be 
	the quotient of $\ratd$ by the group of Möbius transformation acting by conjugacy.
	We will call $\ratd$ \emph{the parameter space} of degree $d$ rational maps, and
	$\modspace$ the \emph{moduli space} of degree $d$ rational maps.
	We denote by $\mathcal{O}(f)$ the orbit of $f$ under the action by conjugacy 
	of the group of Möbius transformations.
\end{defi}

The parameter space $\ratd$ is $2d+1$ dimensional complex manifold, and $\modspace$
is a $2d-2$ complex orbifold.
Denote by $\critf$ the critical set of $f$, i.e. the set of points $z$ where
$Df(z)=0$.
Let 
$\Delta \ni \lambda \mapsto f_\lambda$ be a holomorphic curve in $\ratd$ passing through
$f$ at $\lambda=0$.
Denote by $\dot f$ the section $\frac{df_\lambda}{d\lambda}_{\lambda=0}$ of the line bundle
$f^*T\rs$, and by $\eta$ the meromorphic vector field 
$$\eta(z) := Df^{-1}(z) \cdot \dot f(z).$$

Note that $\eta$ is holomorphic outside of $\critf$, and that 
its poles have order at most the order of the corresponding critical points of $f$.
Denote by $T(f)$ the vector space of meromorphic vector fields on $\rs$ satisfying this property, 
i.e. if $\eta \in T(f)$, then all the poles of $\eta$ are in $\critf$
and the pole at $c \in \critf$ of $\eta$ has order at most the order of $c$ as a critical point of $f$.

Then the map $\dot f \mapsto \eta = Df^{-1} \cdot (\dot f)$ induces a canonical isomorphism
between $T_f \ratd$ and $T(f)$ (indeed, this map is clearly injective and $\dim T(f)=2d+1=\dim \ratd$).

\begin{defi}
	Let $v$ be a summable critical value of a rational map $f: \rs \to \rs$. For any $\eta \in T(f)$, 
	denote by:
	$$\xi_\eta(v) := \sum_{k=0}^{\infty} (f^k)^*\eta(v) \in T_v \rs.$$
\end{defi}

Given an analytic submanifold $\Lambda \subset \ratd$ containing $f$, we say that the critical values
of $f$ move holomorphically on $\Lambda$ if for any critical value $v_i$ of $f$, there are 
holomorphic maps $\lambda \mapsto v_i(\lambda)$ defined on $\Lambda$ such that 
$v_i(\lambda)$ is a critical value of $f_\lambda$. In that case, given a tangent vector $\eta \in T_f\Lambda$, we denote by $\dot v_i$ the derivative at $f$ of $v_i$.

\begin{theoB}[see \cite{levin2011perturbations}]
	Let $f:\rs \to \rs$ be a rational map that is not a flexible Lattès map, with $s$ summable critical values $v_i$, $1 \leq i \leq s$, 
	such that for all $n \in \N^*$, $f^n(v_i) \notin V_f$. 
	Assume that either $f$ has no invariant line field, or that the $\omega$-limit set of those 
	$s$ summable critical values are $C$-compacts. Then there is a germ of analytic set $\Lambda$ in 
	$\ratd$ passing through $f$ and transverse to $\mathcal{O}(f)$ on which the critical values of $f$
	move holomorphically, and such that the linear map
	\begin{align*}
	\mathcal{V} : T_f \Lambda \subset T(f) &\rightarrow \bigoplus_{1 \leq i \leq s} T_{v_i} \rs   \\
	\eta &\mapsto (\dot v_i + \xi_\eta(v_i))_{1 \leq i \leq s}
	\end{align*}
	has maximal rank, i.e. equal to $s$.
\end{theoB}

Let us briefly give an interpretation of this result in the particular case where $f$ is a so-called
Misiurewicz map, meaning that for all $1 \leq i \leq s$, the critical values $v_i$ belong to a uniformly expanding set $K \subset \julia(f)$. In that case, it is a classical result that 
there is a dynamical holomorphic motion of $K$, meaning that there is a family of homeomorphisms 
$h_\lambda: K\to \rs$ such that $h_\lambda \circ f = f_\lambda \circ h_\lambda$ for all $f_\lambda$ 
close enough to $f$. Differentiating $h_\lambda$ with respect to $\lambda$, we obtain the 
relation 
$$\dot h \circ f = \dot f + Df \cdot \dot h,$$
which we may rewrite $f^*\dot h - \dot h = \eta$, where $f^*$ is the action of $f$ by pullback on vector 
fields. Since $f$ is uniformly expanding on $K$, we then obtain
$$\dot h(v_i)=-\sum_{n=0}^\infty (f^n)^*\eta(v_i)=-\xi_\eta(v_i).$$
Therefore, in the case of a Misiurewicz map, Theorem B is a transversality statement concerning the hypersurfaces $H_i:\{ v_i(\lambda) = h_\lambda(v_i) \}$: it implies (together with the submersion lemma)
 that the $H_i$ are smooth and intersect each other transversally at $f$
(that result was first proved by van Strien, see \cite{van2000misiurewicz}). 
In fact, it implies the stronger statement that the functions $v_i(\lambda)-h_\lambda(v_i)$, $1 \leq i \leq s$, 
form near $f$ a partial system of local coordinates.

In the more general setting of Theorem B, such a dynamical motion need not exist, yet 
the vector $\xi_\eta(v_i)$ remains well-defined. 
Another possible interpretation is then to see $\dot v_i + \xi_\eta(v_i)$ as a limit ratio between 
parameter and dynamical derivatives,
as seen from the following formula (see \cite{levin2011perturbations}):
$$\dot v_i + \xi_\eta(v_i) = \limn \frac{\frac{\partial}{\partial \lambda} f_\lambda^n(v_i(\lambda))}{(f^n)'(v_i)}.$$
This formula can be obtained by a direct calculation, noting that $\dot v_i = \dot f(c_i)$, 
where $c_i$ is a critical point such that $f(c_i)=v_i$.

\subsection*{Outline} In section \ref{sec:def}, we will recall some facts about the Teichmüller space and the deformation
space of a finite type map, as well as a description of their cotangent bundle and the immersion
of the Teichmüller space into the deformation space. In sections \ref{sec:ddf} and \ref{sec:injnabla}, we will prove some technical lemmas 
on quadratic differentials that will be useful to the proof of the main theorems. Finally, 
sections \ref{sec:proofA} and \ref{sec:proofB} are devoted to the proofs of Theorem A and Theorem B respectively.

\subsection*{Acknowledgements}The author is indebted to Adam Epstein for helpful conversations.

\section{Deformation spaces}\label{sec:def}

Recall the following objects from Teichmüller theory:

\begin{defi}
	Let $f: W \to X$ be a finite type map, and assume that $W \subset X$.
	Let us define:
	\begin{itemize}
		\item $\qc(f)$ the group of quasiconformal homeomorphisms $\phi$ of $X$ such that 
		$\phi \circ f = f \circ \phi$ wherever this equation is defined
		\item $\qc_0(f)$ the subgroup of $\qc(f)$ of those elements $\phi$ such that there exists
		a uniformly quasiconformal isotopy relatively to the ideal boundary of $X$
		$\phi_t \in \qc(f)$, $0 \leq t \leq 1$, with 
		$\phi_0=\phi$ and $\phi_1$ is the identity on $X$
		\item $\Belf$ is the space of Beltrami forms on $X$ that are invariant under $f$ on $W$,
		and vanish on $X-W$.
		\item The Teichm\"{u}ller space of $f$, denoted by $\teich(f)$, is defined as the quotient
		$\Belf/\qc_0(f)$
		\item $\mathrm{Bel}(X)$ is the space of all Beltrami forms on $X$
		\item $\mathrm{bel}(f)$ is the space of Beltrami differentials on $X$ invariant under $f$
		(a Beltrami differential is a Beltrami form for which we only assume that it has finite $L^\infty$ norm
		instead of norm less than one)
		\item $\mathrm{bel}(X)$ is the space of Beltrami differentials on $X$.
	\end{itemize}
\end{defi}

The reader unfamiliar with Teichmüller theory may find some background in \cite{hubbard2006teichmuller}
and \cite{gardiner2000quasiconformal}.
Notice that the Beltrami form that is identically zero gives a natural basepoint in 
$\teich(f)$.

\begin{defi}
	Let $W$ be a complex $1$-manifold, and let $A \subset W$ be a finite set. We denote by
	$Q(W)$ the space of integrable holomorphic quadratic differentials on $W$, and by
	$Q(W,A)$ the space of integrable meromorphic quadratic differentials on $W$, with at worst simple poles,
	whose poles are in $A$.
\end{defi}

We shall consider the natural action of a finite type map 
$f$ by pushforward on several types of objects.
The pushforward operator will always be denoted by $f_*$, and the pullback by $f^*$,
regardless of the object. Similarly, we will denote by $\nabla_f$ the operator $\id - f_*$, acting on quadratic differentials or on objects that we will consider in Section \ref{sec:ddf}.

If $f: \s_1 \to \s_2$ is a holomorphic map between Riemann surfaces, and $q$ is 
a holomorphic quadratic differential on $\s_2$, then the pullback $f^* q$ defined by
$$f^*q(x;v):= q (f(x); Df(x)\cdot v)$$
for $(x,v) \in T\s_1$ is a holomorphic quadratic differential on $\s_1$. If now 
$f: \s_1 \to \s_2$ is a holomorphic ramified cover, and $q$ is a holomorphic quadratic differential on 
$\s_1$, then the pushforward $f_* q$ is defined as 
$$(f_* q)(x) = \sum_i (g_i^* q)(x)$$
where $x \in \s_2 \backslash S(f)$ and the $g_i$ are the inverse branches of $f$. 
If $f$ has finite degree, the sum is always well-defined and $f_* q$ is a meromorphic quadratic 
differential on $X_2$, with poles at $S(f)$.

One can check that if $q$ is integrable, then $f_* q$ must also be integrable and well-defined even if 
$f$ has infinite degree, and that integrable holomorphic quadratic differentials on a finite type Riemann surface $X \backslash A$
(with $X$ compact and $A \subset X$ finite) are exactly the meromorphic quadratic differentials
on $X$ with at worst simple poles, all of which are in $A$.
Therefore, if $f: W \to X$ is a finite type map with $W \subset X$, the pushforward operator $f_*$ maps $Q(X,A)$ to $Q(X,f(A) \cup S(f))$.

Assume from now on that $f: W \to X$ is a finite type map with $W \subset X$.

\begin{theo}[\cite{epstein1993towers}, Corollary 9 p. 137]\label{th:structteich}
	The space $\teich(f)$ is a finite dimensional complex manifold. 
\end{theo}

\begin{defi}
	Let us denote by $\Lambda_f$ the union of the Julia set of $f$ and of the 
	closure of the grand orbit of $S(f)$ in $X$, and 
	$\Omega_f = X - \Lambda_f$. Also denote by $Q_f$ the space of integrable quadratic differentials on 
	$X$ that are holomorphic on $\Omega_f$.
\end{defi}

Notice that $\Omega_f \cap W$ is an open subset of the Fatou set of $f$; also note that 
when $\Lambda_f$ has zero Lebesgue measure, $Q_f$ is canonically isomorphic to $Q(\Omega_f)$.

\begin{defi}
	Let $f : W \to X$ be a finite type map. Let $A \subset W$ and $B \subset X$ be two finite sets.
	We say that $(A,B)$ is admissible for $f$ if:
	\begin{itemize}
		\item $A \subset B$
		\item $f(A) \subset B$
		\item $S(f) \subset B$
		\item if $X$ has genus $0$, then $\card A \geq 3$, and if $X$ has genus $1$, then 
		$\card A \geq 1$.
	\end{itemize}
\end{defi}

Let $f: W \to X$ be a finite type analytic map, with $W \subset X$.
If $(A,B)$ is admissible for $f$, then $A \subset B$, so we have a natural forgetful map
$\varpi: \teich(X,B) \to \teich(X,A)$.

Moreover, since $f(A) \subset B$ and $S(f) \subset B$, we have a well-defined pullback map
$\sigma_f : \teich(X,B) \to \teich(X,A)$ obtained by pulling back Beltrami forms from $X$ to $W$
using $f$, and then extending them by $0$ to the rest of $X$; this operation on Beltrami forms
descends to a holomorphic map from $\teich(X,B)$ to $\teich(X,A)$.

\begin{defi}[\cite{epstein1993towers}]
	We define $\defab(f)$ by:
	$$\defab(f)=\{ \tau \in \teich(X,B), \varpi(\tau) = \sigma_f(\tau)   \}.$$
\end{defi}

Note that once again, the zero Beltrami form on $X$ induces a natural basepoint in
$\defab(f)$.
From its definition, $\defab(f)$ is clearly an analytic set. But we can in fact say more:

\begin{theo}[\cite{epstein1993towers}]
	Let $f: W \to X$ be a non-exceptional finite type map, and let $(A,B)$ be admissible for $f$.
	Then $\defab(f)$ is a complex manifold of dimension $\card \left(B-A\right)$, whose cotangent space at 
	the basepoint canonically identifies with:
	$$Q(X,B)/\nabla_f Q(X,A).$$
\end{theo}

Let $\mathrm{Bel}(X)$ denote the space of Beltrami forms on $X$.
The identity map $\mathrm{Bel}(X) \to \mathrm{Bel}(X)$ descends to a natural holomorphic map
$i : \teich(f) \to \defab(f)$, mapping basepoint to basepoint.

The next theorem has been proved in \cite{astorg2014dynamical} in the case of rational maps, through
a new construction of the complex structure of $\teich(f)$ that bypasses the use of
certain sophisticated tools from Teichmüller theory. 
A. Epstein has a different (unpublished) proof.

\begin{theo}\label{th:immteichtodef}
	The  cotangent space at the
	basepoint of $\teich(f)$ canonically identifies with
	$Q_f/\overline{\nabla_f Q_f}.$
	Moreover, if $f$ is non-exceptionnal, $(A,B)$ is admissible for $f$ and 
	$B \subset \Lambda_f$, then
	the natural map $\Psi_T: \teich(f) \to \defab(f)$ is an immersion, 
	and the kernel of its codifferential
	at the basepoint is given by:
	$$\ker D\Psi_T(0)^* = \left( Q(X,B) \cap \overline{\nabla_f Q_f} \right)/\nabla_f Q(X,A) .$$
\end{theo}

In the statement of the above theorem, the notation $ \overline{\nabla_f Q_f}$
refers to the $L^1$-closure of the complex vector space ${\nabla_f Q_f}$.
The proof of the particular case where $f$ is a rational map, available in \cite{astorg2014dynamical},
can easily be adapted to the general case of a finite type map. For the convenience of the reader, 
we will include here a sketch of the proof. Note that just like in \cite{astorg2014dynamical},
this approach may also lead to another construction of the complex structure on $\teich(f)$.

\begin{proof}[Proof of Theorem \ref{th:immteichtodef}]
	Let $(A,B)$ be admissible for $f$, with $B \subset \Lambda_f$. The only difference between the proof of 
	Theorem \ref{th:immteichtodef} and that of the main theorem of \cite{astorg2014dynamical} 
	is that we are going to replace the moduli space $\modspace$ of degree $d$ rational maps 
	with the deformation space $\defab(f)$.

	The natural inclusion $\Belf \to \mathrm{Bel}(X)$ descends to a holomorphic 
	map $\Psi: \Belf \to \defab(f)$. 
	
	\begin{lem}\label{lem:rangdpsi}
		The kernel of $D\Psi(0)$ is equal to 
		$$\ker D\Psi(0)=\{\dbar \xi, \xi_{|\Lambda_f}=0\}.$$
	\end{lem}

	\begin{proof}[Proof of Lemma \ref{lem:rangdpsi}]
	The differential of this map at the basepoint
	is the restriction to $\mathrm{bel}(f)$ of the quotient map 
	$$\mathrm{bel}(X) \to T_{[0]} \teich(X,B)$$ 
	whose kernel is $\{\dbar \xi, \xi \,\text{quasiconformal vector field on }X \,\text{s.t }  \xi_{|B}=0 \}$
	(see for example \cite{hubbard2006teichmuller} for a description of $T_{[0]} \teich(X,B)$).
	Since $B \subset \Lambda_f$, if $\mu=\dbar \xi$ with $\xi_{|\Lambda_f}=0$, then 
	$\mu \in \ker D\Psi(0)$.
	Conversely, if for $\dbar \xi \in \mathrm{bel}(f)$ we let $\eta=f^*\xi-\xi$ on $W$,
	 then $\eta$ is a meromorphic vector field on $W$
	(with poles at the critical points of $f$). Indeed, we have $\dbar \xi = f^*\dbar \xi$ so 
	outside of the critical points of $f$, $\dbar \xi = \dbar(f^*\xi)$ so that $\dbar \eta=0$ 
	outside of the critical points of $f$.
	
	Now we claim that if $\mu \in \ker D\Psi(0)$, then $\mu$ can be written as $\dbar \xi$,
	where $\xi$ is a quasiconformal vector field invariant under pullback by $f$ on $W$. Indeed, let $Z$ be 
	any union of repelling periodic points of $f$ disjoint from $B$, of cardinal at least $3$.
	Let $A'=A \cup Z$, and $B'=B \cup Z$; then $(A',B')$ is admissible for $f$. 
	The forgetful map $\teich(X,B') \to \teich(X,B)$ induces a natural holomorphic map
	$\Phi: \mathrm{Def}_{A'}^{B'}(f) \to \defab(f)$; its codifferential at the basepoint is the 
	natural map 
	$$D\Phi(0)^*: Q(B)/\nabla_f Q(A) \to Q(B')/\nabla_f Q(A')$$
	induced by the inclusion $Q(B) \to Q(B')$. This map is injective: indeed, if $q \in Q(B)$
	is of the form $q=\nabla_f  \phi$ with $\phi \in Q(A')$, then $\nabla_f \phi$ does not have 
	any poles at any repelling points in $Z$. So this means that the polar part of $\phi$ on $Z$
	is invariant, but since $Z$ is a union of repelling points this implies that $\phi$ does not 
	have any pole on $Z$ (see \cite{epstein2009transversality}). So in fact $q \in \nabla_f Q(A)$,
	which means that $D\Phi(0)^*$ is injective. Since 
	$$\dim Q(B)/\nabla_f Q(A) = \dim Q(B')/\nabla_f Q(A')=\card (B-A),$$
	$D\Phi(0)$ is also surjective, which implies that $D\Phi(0)$ is injective.
	Now notice that if we let $\Psi'$ be the natural map 
	from $\Belf$ to $\mathrm{Def}_{A'}^{B'}(f)$, then $\Psi = \Phi \circ \Psi'$, so that 
	$D\Psi(0) = D\Phi([0]) \circ D\Psi'(0)$. Since 
	$D\Phi(0)$ is injective, we have therefore $\ker D\Psi(0) = \ker D\Psi'(0)$. Since this is true for any
	choice of $Z$, this means that for any finite union of repelling point $Z$, there is a quasiconformal
	vector field $\xi_Z$ such that $\dbar \xi_Z  =\mu$ and $\xi_Z=0$ on $Z$. For every choice of $Z_1,Z_2$ 
	of cardinal at least $3$, the difference $\xi_{Z_1}-\xi_{Z_2}$ is a holomorphic vector field on $X$ 
	vanishing on at least $3$ points, so by the Riemann-Roch formula $\xi_{Z_1}=\xi_{Z_2}$.
	This means that every $\mu \in \ker D\Psi(0)$ may be written as $\mu=\dbar \xi$ with $\xi$ vanishing
	on \emph{every} repelling periodic point, so $\eta=f^*\xi-\xi$ also vanishes on every repelling 
	periodic point of $f$. Since $\eta$ is meromorphic on $W$, me must have $\eta=0$ by the isolated
	zeroes theorem. So $\xi=f^*\xi$ on $W$, which implies that $\xi_{|\Lambda_f}=0$.
	\end{proof}

	Once we have this description of the kernel of $D\Psi(0)$,
	and using the classification of non-escaping Fatou components
	for finite type maps, the proof of [\cite{astorg2014dynamical}, Th. 4.5]
	carries over verbatim, proving that the differential of $\Psi$ has constant rank on $\Belf$. 
	The map $\Psi: \Belf \to \defab(f)$ descends to the  map $\Psi_T: \teich(f) \to \defab(f)$.
	Applying the constant rank theorem to $\Psi$, we obtain local coordinates on $\Belf$ and $\defab(f)$,
	in which (following Section 5 of \cite{astorg2014dynamical})  the map $\Psi_T$ can be written locally
	as a linear inclusion between two finite dimensional complex spaces. Therefore, the map $\Psi_T$ is 
	an immersion. Let $\pi: \Belf \to \teich(f)$ be the quotient map; since 
	$D\Psi_T ([0]) \circ D\pi(0)=D\Psi(0)$, and since $D\Psi_T(0)$ is injective, we must have 
	$\ker D\pi(0) = \ker D\Psi(0)$. Therefore, using the preceding lemma, we have the following 
	canonical identification:
	$$T_{[0]} \teich(f) = \belf/\ker D\pi(0) = \belf/\{\dbar \xi, \xi_{|\Lambda_f}=0\},$$
	and by duality:
	$$T_{[0]}^* \teich(f) = Q_f/\overline{\nabla_f Q_f}.$$
	
	Finally, it just remains to prove that the kernel of $D\Psi_T(0)^*$ is 
	 $\left( Q(X,B) \cap \overline{\nabla_f Q_f} \right)/\nabla_f Q(X,A)$.
	But since the differential $D\Psi_T(0)$ is the natural map 
	$$\belf/\{\dbar \xi, \xi_{|\Lambda_f}=0\} \to \mathrm{bel}(X)/\{\dbar \xi, \xi_{|B}=0\},$$
	the codifferential $D\Psi_T(0)^*$ is the natural map
	$$Q(X,B)/\nabla_f Q(X,A) \to Q_f/\overline{\nabla_f Q_f},$$
	whose kernel is clearly $\left( Q(X,B) \cap \overline{\nabla_f Q_f} \right)/\nabla_f Q(X,A)$.
\end{proof}

%\begin{proof}
%	The differential of the map $i$ at the basepoint lifts to the identity map 
%	$\mathrm{bel}(X) \to \mathrm{bel}(X)$; its codifferential is the natural map
%	$$Di(0)^*: Q(X,B)/\nabla_f Q(X,A) \to Q_f/\overline{\nabla_f Q_f},$$
%	that lifts to the natural inclusion $Q(X,B) \to Q_f$.
%	Therefore its kernel is indeed
%	$$\left( Q(X,B) \cap \overline{\nabla_f Q_f} \right) /\nabla_f Q(X,A).$$
%	
%	Now let us prove that $Di(0)^*$ is surjective, which will prove that $i$
%	is a submersion.
%	
%	Let $q \in Q_f$: we shall prove that there exists $\phi \in Q(X,B)$ and $\psi \in \overline{\nabla_f Q_f}$
%	such that $q = \phi + \psi$. Let $(Z_n)_{n \in \N}$ be an increasing collection of repelling cycles 
%	that are dense in $\Lambda_f$ 
%\end{proof}

\section{Action of quadratic differentials on vector fields}\label{sec:ddf}

In this section, $\s$ will denote a compact Riemann surface  of genus $g$.
If we chose an arbitrary Hermitian metric on $\s$, we get a topology on the space
$\Gamma(T\s)$ of continuous vector fields on $\s$, induced by the norm
$$\|\xi\|_{\infty}= \sup_{s \in \s} \|\xi(s)\|.$$
This norm depends on the particular choice of the Hermitian metric, but not the topology it induces
(by compacity of $\s$). We will refer to it as the \emph{uniform topology} for continuous vector fields
on $\s$.

\begin{defi}
	Denote by $\ddf$ the (topological) dual of the topological vector space of continuous vector fields
	on $\s$, equipped the topology dual to the uniform topology.
\end{defi}

Again, the choice of a Hermitian metric on $X$ gives by duality a norm generating the topology
on $\ddf$, but that topology is independant from the choice of the norm. Depending on the genus $g$ of $X$,
it will be convenient to use different choices of metrics in the following proofs.

\begin{defi}\label{def:diffquadreg}
	Let $q$ be an integrable quadratic differential on $\s$. Then $q$ induces a linear form on the space
	of smooth vector fields in the following way:
	$$\xi \mapsto \int_{\s} q \cdot \dbar \xi.$$
	If that linear form extends continuously to an element of $\ddf$, we denote
	that extension by $\dbar q$ and we say that $q$ is \emph{regular}.
\end{defi}

%Notons que si $q$ s'écrit en coordonnées $q=h(z)dz^2$, alors $q$ est 
%régulière si et seulement si $\frac{\partial h}{\partial \overline{z}}$ (au sens des distributions) 
%est une mesure complexe de Radon. C'est le cas en particulier pour toutes
%les différentielles quadratiques méromorphes à pôles simples, dont le 
%$\dbar$ est une somme de diracs.
%
%Une conséquence immédiate du lemme de Weyl est que si $q$ est une différentielle 
%quadratique régulière telle que $\dbar q$ a un support inclus dans un compact 
%$K$, alors $q$ est holomorphe en dehors de $K$.

Note that if  $q$ is written in local coordinates as $q=h(z)dz^2$, then $q$ is 
regular if and only if $\frac{\partial h}{\partial \overline{z}}$ (in the sense of distributions) 
is a complex Radon measure. It is in particular the case when $q$ is meromorphic 
with at worst simple poles, in which case $\dbar q$ has finite support.

An immediate consequence of Weyl's lemma is that if $q$ is a regular 
quadratic differential such that $\dbar q$ is supported in a compact $K$,
then $q$ is holomorphic outside of $K$.

\begin{prop}\label{prop:struct}
	Let $\mcal$ be the space of Radon measures on $\s$, and $A = C^0(\s, \C)$.
	Let $\Omega^{1,0}(X)$ denote the space of complex-valued continuous 
	forms of bidegree $(1,0)$ on $X$.
	The map
	\begin{align*}
	\mcal \otimes_A \Omega^{1,0}(\s) &\rightarrow \Gamma(T\s)^* \\
	\mu \otimes \alpha &\mapsto \left(\xi \mapsto \int_{\s} \alpha(\xi) d\mu \right)     
	\end{align*}
	is an isomorphism of $A$-modules (and therefore of $\C$-vector spaces).
\end{prop}

\begin{rem}
	Since $\Omega^{1,0}(\s)$ is an $A$-module of rank $1$,
	every element of $\Omega^{1,0}(\s) \otimes_A \mcal$ can be written as
	$\alpha \otimes_A \mu$, where $\alpha \in \Omega^{1,0}(\s)$ and $\mu \in \mcal$.
\end{rem}

\begin{proof}
	The considered map is clearly an injective morphism of $A$-modules.
	It is therefore enough to prove that it is surjective. Let $u \in \ddf$.
	If the support of $u$ in included in a local coordinate domain $(U,z)$,
	then it is a consequence of Riesz's representation theorem that $u$ can be
	written as $u=dz \otimes_A \mu$, where $\mu$ is a Radon measure 
	of support included in $U$. We conclude easily using a partition of unity.
\end{proof}

We will therefore identify from now on $\ddf$ with $\Omega^{1,0}(\rs) \otimes_A \mcal$.

\begin{defi}
	Let $f: W \to \s$ be an open holomorphic map and let $u = \alpha \otimes \mu \in \ddf$ be such that
	$\frac{\| \alpha \|}{\|Df\|} \in L^1(|\mu|)$ (for any continuous Hermitian metric on $\s$).
	We define the pushforward of $u$, denoted by $f_*u$, by :
	$$\langle f_*u, \xi \rangle := \langle u, f^* \xi \rangle = \int_{\rs} \alpha(f^*\xi)d\mu.$$
\end{defi}

Note that $f_*u \in \ddf$.
In particular, if $u \in \ddf$ has support $K$ that does not meet
 $S(f)$, then $f_* u$ is well-defined and has a support included in
$f(K)$.
%Note that the linear  $f_*$ n'est pas continu, et n'est défini que sur un 
%sous-espace dense de $\ddf$.

\begin{lem}\label{lem:resolutiondbar}
	Let $Z \subset \s$ be a subset of cardinal $|3g-3|$. Let $u \in \ddf$ be supported in
	$\{y\}$, where $y \in X$.
	Then there is a unique meromorphic quadratic differential $q$ on $\s$ with at worst simple poles such that:
	\begin{itemize}
		\item if $g=0$, then 
		 $\dbar q - u$ is supported in $Z$ 
		\item if $g\geq 1$,  $\dbar q = u$ and for all $z \in Z$, $q(z)=0$.
	\end{itemize}
	Moreover, for any choice of Hermitian metric on $\s$, there is a constant $C>0$ depending only on that 
	metric and on $Z$ such that 
	$\|q\|_{L^1} \leq C \|u\|_{\infty}.$
\end{lem}

\begin{proof}
	We will treat separately the three following cases: $g=0$, $g=1$ and $g \geq 2$.
	
	\subsubsection*{The case of genus $0$}
	\medskip
	If $X$ has genus $0$, then $X$ is isomorphic to the Riemann sphere $\rs$. Note that 
	a meromorphic quadratic differential $q$ with at worst simple poles on $X$ will satisfy
	the property that $\dbar q - u$ if and only if $q$ has at worst four simple poles, located in 
	$Z \cup \{y\}$, and for all smooth vector fields $\xi$ vanishing on $Z$, 
	$$\int_X q \cdot \dbar \xi = \langle u, \xi \rangle.$$

	If we work in affine coordinates in which $Z=\{0,1,\infty\}$, then $q$ has the form:
	$$q(z) = \alpha \frac{y(y-1)}{z(z-1)(z-y)} dz^2$$
	where $\alpha$ is such that $u=\delta_y \otimes (\alpha dz)$, $\delta_y$ being the Dirac mass
	at $y$. 
	Up to permuting the order of points in $Z$, we may assume that $|y|<1$. 
	Then it is easy to see that $\|q\| \leq C |\alpha|$ for some constant $C>0$ (depending on the 
	coordinates $z$ and therefore on $Z$)
	and that $\|u\| = |\alpha| \sup_{|y|<1} \|dz\|$.
	Therefore there is a constant $C_2>0$ depending only on the metric and on $Z$ such that
	$\|q\|_{L^1} \leq C_2 \|u\|$.	
	This concludes the case of genus $0$.
	\subsubsection*{The case of genus $1$}
	\medskip
	In this case, $Z$ is empty, so we need to prove that there is a unique quadratic differential
	with at worst simple poles such that $\dbar q=u$. Any such quadratic differential must have 
	at worst one simple pole, located at $y$. According to the Riemann-Roch formula, such quadratic 
	differentials form a complex vector space of dimension one. Moreover, Stokes' theorem 
	implies that for all smooth vector fields $\xi$ on $X$,
	$$\langle \dbar q, \xi \rangle = \int_X q \cdot \dbar \xi = 2i\pi \res(q \cdot \xi, y).$$
	Therefore, there is exactly one choice of polar part at $y$ (and therefore exactly one choice of $q$)
	such that $\dbar q = u$. Let us now prove the inequality.
	If $g=1$, then $X$ is a complex torus $\C/\Lambda$, and for any $y \in X$, there is a translation 
	descending to an automorphism $T_y$ of $X$ mapping the basepoint $[0]$ to $y$. 
	The pullback 
	map $T_y^*$ induces an isometry for the $L^1$ norm of integrable quadratic differentials, as well as
	for linear forms in $\ddf$ (endowed with the norm induced by the flat Hermitian metric on $X$). 
	In other words, we lose no generality in assuming that $y=[0]$. Then the desired inequality is trivial,
	since the map $u \mapsto q$ is a complex linear map between finite dimensional normed vector spaces 
	(in fact one-dimensional), therefore is continuous.
	
	\subsubsection*{The case of genus at least $2$}
	\medskip
	The existence and unicity is similar to the previous case: notice that if $\dbar q=u$ and $q$ has
	at worst simple poles, then $q$ must have at worst a simple pole at $y$ and must vanish on $Z$. 
	According to the Riemann-Roch formula, such quadratic differentials form a vector space of complex
	dimension one, and once again, the choice of the right polar part at $y$  uniquely determines $q$.
	
	Now let us prove the desired inequality. Since $g \geq 2$, $X$ is hyperbolic, so we may pick 
	its hyperbolic metric as a choice of Hermitian metric inducing a norm on $\ddf$.
	We will work by duality. According to Theorem A in \cite{astorg2014dynamical}, for any quasiconformal vector field $\xi$ on $X$, we have
	$\|\xi\| \leq 4 \|\dbar \xi \|_{\infty}$ (here $\|\xi\|$ is the supremum of the length of $\xi$
	is the hyperbolic metric on $X$, which is finite since $X$ is compact).
	Therefore $\|q \|_{L^1} \leq 4 \|u\|$.
\end{proof}

\begin{theo}\label{th:resolutiondbar}
	Let $Z \subset \s$ be a subset of cardinal $|3g-3|$. Let $u \in \ddf$.
	Then there is a unique regular quadratic differential $q$ on $\s$ such that:
	\begin{itemize}
		\item if $g=0$, then 
		$\dbar q - u$ is supported in $Z$ 
		\item if $g\geq 1$,  $\dbar q = u$ and for all $z \in Z$, $q(z)=0$.
	\end{itemize}
	Moreover, for any choice of Hermitian metric on $\s$, there is a constant $C>0$ depending only on that 
	metric such that 
	$\|q\|_{L^1} \leq C \|u\|_{\infty}.$
\end{theo}

We will say that the quadratic differential $q$ given by the above theorem is \emph{the}
$Z$-normalized quadratic differential corresponding to $u$.

\begin{proof}
	According to Proposition \ref{prop:struct}, we can write $u=\mu \otimes \alpha$. 
	For any $y \in X$, let $u_y = \delta_y \otimes \alpha$, where $\delta_y$ is the Dirac mass at $y$,
	and $q_y$ be the corresponding quadratic differential given by the preceding lemma. Let 
	$r_y=\dbar q_y - u_y$: by definition, $r_y$ is supported in $Z$.
	The second part of that lemma implies that there is a constant $C>0$ depending only on
	 the choice of Hermitian metric and on $\alpha$
	such that for all $y \in X$, $\|q_y \|_{L^1} \leq C$. Let $q= \int_X q_y \, \mu(y)$.
	Note that $q$ is integrable and $\|q \|_{L^1} \leq C$.
	We will prove that $q$ satisfies the desired property. Let $\xi$ be a smooth vector field on $X$. 
	We have:
	\begin{align*}
		\int_X q \cdot \dbar \xi &= \int_X \left( \int q_y \,\mu(y)\right) \cdot \dbar \xi  \\
								&= \int \left( \int_X q_y \cdot \dbar \xi\right) \mu(y)  \\
								&= \int \langle u_y, \dbar \xi \rangle \mu(y)\\
								&= \int \alpha_y(\xi) \, \mu(y) + \int \langle r_y, \xi \rangle \mu(y)\\
								&=  \langle u, \xi \rangle + \langle r, \xi \rangle							
 	\end{align*}
 	where $\langle r, \xi \rangle := \int  \langle r_y, \xi \rangle \mu(y)$ is supported in $Z$ if $g=0$,
 	and is identically zero otherwise.
 	Thus the theorem is proved.
\end{proof}

\begin{prop}\label{prop:supportfq}
	Let $q$ be a regular quadratic differential, and $f: W \to \s$ a finite type analytic map.
	Assume that $f_* \dbar q$ and $\dbar f_*q$ are well-defined as elements of $\ddf$.
	Then
	$$\mathrm{supp}\,(\dbar f_* q - f_* \dbar q) \subset S(f).$$ 
\end{prop}

\begin{proof}
	Let $\xi$ be a quasiconformal vector field vanishing on a neighborhood of $S(f)$.
	Then $f^*\xi$ is also quasiconformal  (and it vanishes in the neighborhood of 
	$\critf$), and $\dbar f^*\xi = f^*\dbar \xi$.
	Therefore:
	$$\langle \dbar f_*q, \xi \rangle = \langle f_* \dbar q, \xi \rangle,$$
	and so $\langle f_* \dbar q - \dbar f_*q, \xi \rangle =0$.
	This exactly means that  $\mathrm{supp}\,(\dbar f_* q - f_* \dbar q) \subset S(f)$. 
\end{proof}

\section{Extended infinitesimal Thurston rigidity}\label{sec:injnabla}

\begin{defi}
	A compact set $K \subset \s$ is called a $C$-compact if it satisfies the following property: 
	 any continuous function
	 on $K$  can be uniformly approximated  by restrictions of 
	functions that are holomorphic  on  a neighborhood of $K$.
\end{defi}

Note that a $C$-compact must have empty interior.
The following proposition gives sufficient conditions
for a
compact to be a $C$-compact. The proof is adapted from
\cite{makienko2008remarks} to the case of a general Riemann surface.

\begin{rem}\label{rem:localization}
	In fact, it can be proved (see \cite[Th. 2]{boivin2004uniform}) that being a $C$-compact is a local property, 
	in the sense that $K$ is a $C$-compact if and only if for every point $p \in K$,
	there is a basis of neighborhoods $(U_n)_{n \in \N}$ such that $K \cap \overline{U_n}$
	is a $C$-compact. Therefore we can replace functions by vector fields
	(or sections of any holomorphic line bundle) without 
	changing the definition of $C$-compact.
\end{rem}

\begin{prop}\label{prop:ccompact}
	Let $K$ be a compact subset of $\s$. Each of the following properties imply that 
	$K$ is a $C$-compact :
	\begin{itemize}
		\item[$i)$] $K$ has zero Lebesgue measure, or
		\item[$ii)$] $K$ has empty interior and disconnects $\s$ into finitely many connected components.
	\end{itemize}
\end{prop}

\begin{proof}
	Those conditions have been observed to imply that $K$ is a $C$-compact 
	in \cite{makienko2005remarks} and \cite{levin2011perturbations} in the 
	case where $X=\rs$. The following are immediate adaptations to the general case of a 
	compact Riemann surface $X$.
	\begin{itemize}
		\item[$i)$] This follows from the local nature of being a $C$-compact
		(remark \ref{rem:localization}) and Vitushkin's theorem 
		(see e.g. \cite{gamelin2005uniform}).
		\item[$ii)$] This follows from \cite{scheinberg1978uniform}, by taking $M$ 
		to be $X$ with a closed disk removed from every connected component of $X-K$.
	\end{itemize}
\end{proof}

\begin{theo}\label{th:makccomp}
	Let $K \subset \s$ be a $C$-compact, and let $q$ be a regular quadratic differential supported in $K$.
	Then $q=0$ Lebesgue a.e.
\end{theo}

\begin{proof}
	Let $q$ be a regular quadratic differential supported in a $C$-compact $K$. We shall prove that
	$\dbar q=0$ as an element of $\Gamma(TX)^*$. By definition of a $C$-compact and by Remark
	\ref{rem:localization}, any continuous vector field can be uniformly approximated on $K$ by
	restrictions of vector fields that are holomorphic in the neighborhood of $K$, so it is enough
	to test $\dbar q$ against such vector fields. Let $\xi$ be a smooth vector field on 
	$\s$ that is holomorphic on a neighborhood $U$ of $K$. Then:	
	\begin{align*}
	\langle \dbar q, \xi \rangle &= \int_{\s} q \cdot \dbar \xi \\
	\langle \dbar q, \xi \rangle &= \int_{U} q \cdot \dbar \xi + \int_{\s - U} q \cdot \dbar \xi
	\end{align*}
	Since $q$ is supported in $K$, we have $\int_{\s - U} q \cdot \dbar \xi=0$.
	Since $\xi$ is holomorphic on $U$, we have $\int_{U} q \cdot \dbar \xi =0$.
	Therefore $\langle \dbar q, \xi \rangle =0$ and so $\dbar q=0$.
	So by Weyl's lemma,
	$q$ is holomorphic on $\s$ (up to a set of Lebesgue measure zero), and vanishes on the 
	(non-empty) open set $\s-K$, so 
	$q=0$ Lebesgue-a.e.
\end{proof}

Recall the following fundamental fact:

\begin{prop}[see \cite{epstein1993towers}, Corollary 8 p. 124]\label{prop:injectnablaf}
	Let $f: W \to X$ be a non-exceptional finite type analytic map, with $W \subset X$.
	Let $A \subset X$ be a finite set. Then if $q \in Q(X,A)$ and $q=f_*q$, then $q=0$.
\end{prop}

We will now investigate what happens if we relax the assumption that $q$ is meromorphic.
The following result will be needed:

\begin{theo}[\cite{epstein1993towers}]\label{th:classification}
	Let $U$ be a non-escaping Fatou component for $f$. Then $U$ is eventually mapped to a periodic 
	component which is
	is either an attracting basin, a parabolic basin, a Herman ring or a Siegel disk.
\end{theo}

\begin{defi}
	A rotation annulus for $f$ is 
	a connected component of $\Omega_f$ which is an annulus of finite modulus and on which 
	the dynamics of $f$ is conjugate to an irrational rotation.
	
	A cycle of rotation annuli for $f$ of period $p$
	is a family of components $(\acal, \ldots, f^{p-1}(\acal))$ of $\Omega_f$
	which are all rotation annuli for $f^p$.
\end{defi}

%Ces anneaux de rotations sont exactement les anneaux dans les disques de Siegel et les anneaux de Herman 
%qui sont délimitées par des adhérences d'orbites critiques.

%On peut associer à chaque anneau de rotation $A$ une différentielle quadratique 
%$q \in Q(\Omega_f)$ invariante
%par $f$ de la manière suivante: soit $\phi$ une coordonnée linéarisante pour $f$ sur $A$, 
%qui redresse $A$ sur un anneau droit $A_R =\{1 <|z| < R\}$.
%Soit $\tilde q \in Q(A_R)$: alors $q =\phi^* \tilde q$ est invariante par $f$ si et seulement si
%$\tilde q$ est invariante par rotation. 
%On vérifie sans problème que $\tilde q = \frac{dz^2}{z^2}$ est invariante par rotation; de plus, 
%$\tilde q$ est intégrable sur $A_R$. Donc $q = \phi^* \tilde q$ (étendue par $0$ en dehors de 
%$A$) est invariante et intégrable.

To each rotation annulus $\acal$, we may canonically associate a quadratic differential $q_\acal$
in the following way: let $\phi : \acal \to \C$ be a linearizing coordinate for $f$ on $\acal$, mapping
$\acal$ to a straight annulus $A(R)=\{1 < |z| < R\}$. 
Let 
\begin{equation}
	q_\acal = \phi^* \left(\frac{dz^2}{z^2}\right).
\end{equation}

One can easily check using Laurent series 
that $\frac{dz^2}{z^2}$ is up to scalar multiplication the only holomorphic quadratic differential on 
$Q(A(R))$ that is rotation-invariant: in particular, there are no 
rotation invariant quadratic differential that are integrable near $0$.
 Therefore $q_\acal$ is the only integrable
holomorphic quadratic differential on $Q(\acal)$ that is forward-invariant under $f$.
We can extend it by zero outside of $\acal$ to obtain a forward invariant quadratic differential
in $Q(\Omega_f)$, that we still denote by $q_\acal$.

Similarly, if $(A, \ldots, f^{p-1}(A))$ is a cycle of rotation rings for $f$,
then we get a quadratic differential $\tilde{q}_\acal \in Q(\acal)$ that is invariant under $f^p$.
It is then easy to check that $q_\acal := \sum_{k=0}^{p-1} f_*^k \tilde{q}_\acal$ 
is forward invariant under $f$, and it is (up to scalar multiplication) the only one on $Q(\acal \cup \ldots \cup f^{p-1}(\acal))$.

%De plus, à une constante multiplicative près, c'est la seule: en effet, on peut développer
%tout élément de $Q(A_R)$ en série de Laurent:
%$$\tilde q = \left(\sum_{n \in \Z} a_n z^n \right)dz^2$$
%et l'invariante par rotation s'écrit $\tilde q(e^{i \alpha}z) e^{2i\alpha} = \tilde q(z)$, d'où l'on tire:
%$$\sum_{z \in \Z} a_n e^{2i\alpha} z^n e^{in\alpha} = \sum_{n \in \Z} a_n z^n$$
%et donc $a_n=0$ pour tout $n \neq -2$.
%En particulier, on voit qu'il n'existe pas de différentielle quadratique 
%invariante par rotation intégrable au voisinage de $0$.
%
%
%Dans le cas d'un cycle d'anneaux de rotation $(A, \ldots, f^p(A))$ 
%, le raisonnement précédent fournit une différentielle quadratique $q \in Q(A)$
%invariante par $f^p$. On vérifie sans problème que $Q= \sum_{k=1}^{p} f_*^k q$
%est invariante par $f$ (et elle n'est pas nulle puisque les supports 
%des $f_*^k q$ sont essentiellement disjoints). \\

\begin{prop}\label{prop:suppjuliadiffquadinv}
	Let $f: W \to X$ be a finite type analytic map, with $W \subset X$. Then the only quadratic differentials
	on $Q(\Omega_f)$ invariant by $f_*$ are those described above.
\end{prop}

\begin{proof}
%	Soit $q \in Q(\Omega_f)$  une différentielle quadratique invarinte. Alors $|q|$ est une mesure invariante
%	sur $\Omega_f \subset \fatou(f)$, 
%	à densité par rapport à la mesure de Lebesgue.
%	Soit $U$ une composante de Fatou périodique 
%	dont la grande orbite a une masse positive pour $|q|$. Alors $U$ ne peut pas être une 
%	composante attractive, superattractive ou parabolique, puisque la seule mesure invariante sur 
%	la grande orbite d'une telle composante est une somme de diracs le long du cycle correspondant.
%	
%	Il ne reste donc que la possibilité d'un disque de Siegel ou d'un anneau de Herman.
%	Une telle composante de Fatou n'est jamais complètement invariante
%	pour des raisons de degré local. Soit $V$ une composante 
%	de $f^{-1}(U)$ qui ne fait pas partie du cycle de composantes auquel appartient $U$.
%	L'invariance de $q$ impose que $\int_{f^{-n}(V)} |q|$ soit égale 
%	à $\int_{V} |q|$; or les 
%	ouverts $f^{-n}(V)$ sont tous disjoints.
%	Donc $\int_V |q|=0$ puisqu'on a supposé $q$ intégrable.
%	
%	Ainsi, $q$ s'annule partout sauf éventuellement sur un cycle de disques de Siegel ou d'anneaux
%	de Herman. L'analyse précédente montre que les seules différentielles quadratiques invariantes
%	intégrables que l'on peut obtenir sont alors celles associées à un cycle d'anneaux de rotation. 

	Let $q \in Q(\Omega_f)$  be an invariant quadratic differential. Then $|q|$ is an invariant measure
	on $\Omega_f$, that does not charge any escaping Fatou component.  
		
	Let $U$ be a component of $\Omega_f$ with positive mass for $|q|$, and let $V$ be 
	the non-escaping Fatou component containing $U$.
	According to Theorem \ref{th:classification}, $V$ is eventually mapped to  an attracting 
	basin, a parabolic basin, a Siegel disk or a Herman ring.
	If $V$ is mapped to an attracting or parabolic basin, then every point in $U$ converges
	to the same finite cycle of points, so the grand orbit of $V$ cannot support an invariant measure
	absolutely continuous with respect to the Lebesgue measure. Therefore $V$ must be eventually mapped to 
	either a (periodic) Siegel disk or Herman ring. Such a Fatou component can never be completely invariant,
	since it maps to itself with degree $1$. So $V$ must be in fact in the cycle: indeed, if it were not 
	the case, then the preimages $f^{-n}(V)$ would form a pairwise disjoint family of open sets, each 
	having the same mass as $V$ for $|q|$; but this would contradict the fact that $q$ is integrable.
	
	Therefore $U$ must be in a periodic rotation domain. 
	But by the preceding discussion, the only invariant quadratic differentials on such domains
	are the quadratic differentials $q_\acal$ associated to rotation annuli.
\end{proof}

%
%\begin{theo}[Makienko]\label{th:injectivitereg}
%	Soit $f$ une fraction rationnelle qui n'est pas un exemple de Lattès flexible, 
%	et $K$ un $C$-compact invariant inclus dans son ensemble de Julia.
%	Alors si $q \in Q(K)$ est régulière et invariante par $f$, et si $q$ s'annule
%	sur tout anneau de Herman éventuel de $f$, 
%	on a $q=0$.
%\end{theo}
%
%
%\begin{proof}
%	Soit $q \in Q(K)$ régulière et invariante. La proposition
%	\ref{prop:suppjuliadiffquadinv} implique que le support de $q$ doit être inclus dans l'ensemble de Julia
%	$\julia(f)$. Si $\julia(f) \neq \rs$, alors $q$ doit s'annuler sur $\julia(f) - K$ d'après le principe des zeros isolés. 
%	Comme $K$ est un $C$-compact, le théorème \ref{th:suppccompact} implique que $q=0$.
%	
%	Dans le cas où $\julia(f) = \rs$, si $q \in Q(K)$ est une différentielle quadratique intégrable invariante, alors
%	ou bien $q$ s'annule sur $\rs-K$ ou bien $q$ a des zéros discrets dans $\rs-K$. Dans le premier cas, le théorème
%	\ref{th:suppccompact} implique encore que $q=0$. Dans le second cas, $\frac{\overline{q}}{|q|}$ définit 
%	un champ de droite invariant, et $q$ est holomorphe sur l'ouvert non vide $\rs-K \subset \julia(f)$. 
%	Alors le lemme 3.16 (p. 49) de \cite{mcmullen1994complex} implique 
%	que $f$ est un exemple de Lattès flexible.
%\end{proof}

\section{Proof of Theorem A}\label{sec:proofA}

\begin{defi}
	Let $f: W \to X$ be a finite type map, with $W \subset X$. 
	\begin{itemize}
		\item Let $p(f)$ denote the number of singular values with a periodic or preperiodic orbit
		\item Let $s(f)$ denote the number of summable singular values with an infinite forward orbit,
		whose $\omega$-limit sets are $C$-compacts.
		\end{itemize}
\end{defi}

\begin{theoA}
	Let $f: W \to X$ be a non-exceptional finite type analytic map, with $W \subset X$.
	We have: $$\dim \teich(f) \leq \card S(f) - p(f) - s(f).$$
\end{theoA}

\begin{proof}
	Since repelling periodic points are dense in the Julia set of $f$ (see \cite{epstein1993towers}), we 
	may chose a finite union of repelling cycles of $f$ of cardinal at least $|3g-3|$, 
	if $g \neq 1$, or of cardinal at least one if $g=1$. Let $Z$ denote such a set.
	Let $S^0(f)$ denote the set of singular values of $f$ which are periodic or preperiodic,
	and let 
	$$A=\left(\bigcup_{n \in \N} f^n(S^0(f)) \right)\cup Z.$$
	Let $B=S(f)\cup A$.
	Then $(A,B)$ are admissible, and therefore $\defab(f)$ is a complex manifold of 
	dimension $\card (B-A) = \card (S(f)) - p(f)$.
	
	Let $s \in S(f)- S^0(f)$ be a summable singular value whose $\omega$-limit set 
	is a $C$-compact.
	Let $u \in \ddf$ be a non-zero linear form with support equal to $\{s\}$. Let 
	$v_n = \sum_{k=0}^{n} f_*^k u$. The pushforwards are well-defined, since by 
	definition of summability, the orbit of $s$ does not meet critical points.
	The fact that $s$ is summable readily implies that the sequence $(v_n)_{n \in \N}$ 
	converges to some $v_s \in \ddf$. 
	Let $q_s$ be the $Z$-normalized quadratic differential on $X$ corresponding to $v_s$ 
	(see Theorem \ref{th:resolutiondbar}).
	From the definition of $v_s$, we have that $v_s - f_* v_s = u$; moreover,
	$v_s - \dbar q_s$ is supported in $Z$ (if $g=0$, or $v_s= \dbar q_s$ if $g>0$).
	Therefore, in view of Proposition \ref{prop:supportfq}, $q_s - f_* q_s$ is supported in
	$Z \cup S(f)$ (and in fact, if $g>0$, it is supported only in $S(f)$).
	Let us denote by $S^s(f)$ the set of summable singular values of $f$ whose $\omega$-limit
	sets are $C$-compacts.
	
	\begin{lem}\label{lem:qiindep}
		The quadratic differentials $(\nabla_f q_s)_{s \in S^s(f)}$ are linearly independant.
	\end{lem}

	\begin{proof}[Proof of Lemma \ref{lem:qiindep}]
		First note that the quadratic differentials $(q_s)_{s \in S^s(f)}$ are linearly
		independant, since the $\dbar q_s$ are. Next, we  prove that the vector space spanned by their restriction 
		to $\Omega_f$ is in direct sum
		with the kernel of the operator $\nabla_f : Q(\Omega_f) \to Q(\Omega_f)$.
		According to Proposition \ref{prop:suppjuliadiffquadinv}, we just need to prove that no non-trivial
		element of
		the vector space spanned by the $q_s$, $s \in S^s(f)$, can be written 
		as $\lambda \frac{dz^2}{z^2}$ in linearizing coordinates in a rotation annulus of $f$, 
		with $\lambda \neq 0$ (we may assume that the rotation annulus is fixed, up to replacing $f$
		by one of its iterates).
		So let $q=\sum_{s \in S^s(f)} \lambda_s q_s$, and let $q_\acal = \frac{dz^2}{z^2}$,
		where $z$ is a linearizing coordinate for a rotation annulus. 
		Notice that $\dbar q = \sum_{s \in S^s(f)} \lambda_s v_s$ is of the form $\mu \otimes \alpha$,
		where $\mu$ is a converging series of Dirac masses at points of the post-singular set of $f$.
		In particular, $\mu$ is a measure of dimension $0$.
		On the other hand, let us compute $\dbar q_\acal$.
		Let $\xi$ be a smooth vector field on $X$. Let us denote by $A$ the rotation annulus
		on which $q_\acal$ is supported, and let $\frac{1}{2\pi}\log R$ be its module. Then, by working in the 
		$z$-coordinates, we get:
		\begin{align*}
			\int_X q_\acal \cdot \dbar \xi &= \int_A q_\acal \cdot \dbar \xi \\
								&= \int_{1< |z|<R} \frac{dz^2}{z^2} \cdot \dbar \xi(z).
		\end{align*}
		Let $\epsilon>0$. 
		Notice that Stokes' theorem implies that:
		\begin{align*}
			\int_{1+\epsilon < |z| < R-\epsilon} \frac{dz^2}{z^2} \cdot \dbar \xi(z)
			= \int_{ |z| = R-\epsilon} \frac{dz^2}{z^2} \cdot \xi(z) - \int_{ |z| = 1+\epsilon} \frac{dz^2}{z^2} \cdot \xi(z).
		\end{align*}
		Moreover, since $\phi \cdot \dbar \xi$ is integrable, Lebesgue's dominated convergence theorem
		implies that 
		$$\lim_{\epsilon \to 0} \int_{1+\epsilon < |z| < R-\epsilon} \frac{dz^2}{z^2} \cdot \dbar \xi(z)
		 =  \int_{1< |z|<R} \frac{dz^2}{z^2} \cdot \dbar \xi(z),$$
		 so that 
		 $$\int_X \phi \cdot \dbar \xi = \lim_{\epsilon \to 0} \left(\int_{ |z| = R-\epsilon} \frac{dz^2}{z^2} \cdot \xi(z) - \int_{ |z| = 1+\epsilon} \frac{dz^2}{z^2} \cdot \xi(z) \right).$$
		 Since the pullback of the Lebesgue measure on the circles $|z|=1+\epsilon$ and $|z|=R-\epsilon$
		 converge to the harmonic measure of the closure $\overline{A}$ of $A$ when
		 $\epsilon$ tends to $0$, this means that 
		 $q_\acal$ is regular in the sense of Definition \ref{def:diffquadreg}, and that 
		 $\dbar q_\acal$ is of the form $\nu \otimes \beta$, where $\nu$ is a measure that is absolutely continuous
		 with respect to the harmonic measure of $\overline{A}$. Since the harmonic measure never has dimension
		 $0$, $\dbar q$ can never be a non-zero multiple of $\dbar \phi$, which concludes the proof that 
		 $q$ cannot be a non-zero multiple of $\phi$.
		 
		 This means that if $q$ is invariant under $f$, then $q$ must vanish on $\Omega_f$. Moreover,
		 $q$ is holomorphic outside of the union of the closure of the orbit of $S^s(f)$, so 
		 in fact $q$ vanishes outside of the union of the closure of the orbit of $S^s(f)$, which is a 
		 $C$-compact. Therefore, by Theorem \ref{th:makccomp}, $q=0$. This concludes the proof of the lemma.
	\end{proof}

	Let us now return to the proof of Theorem A.
	Recall that by construction, for any $s \in S^s(f)$, $\nabla_f q_s \in Q(X,B)$.
	It is a consequence of the previous lemma that the classes $([\nabla_f q_s])_{s \in S^s(f)}$
	are linearly independent in $Q(X,B)/\nabla_f Q(X,A)$: indeed, $\nabla_f$ is injective on 
	the vector space spanned by the $q_s, s \in S^s(f)$ and by $Q(X,A)$. Therefore no non-trivial linear 
	combination of the $\nabla_f q_s$ can be in $\nabla_f Q(X,A)$, since none of the $q_s$ are in
	$Q(X,A)$.
	This means that $\dim  \left( Q(X,B) \cap \overline{\nabla_f Q_f}\right)/\nabla_f Q(X,A) \geq s(f)$.
	But by Theorem \ref{th:immteichtodef}: 
	\begin{align*}
		\dim T_{[0]}^*\teich(f) = \dim T_{[0]}^*\defab(f) - \dim  \left( Q(X,B) \cap \overline{\nabla_f Q_f}\right)/\nabla_f Q(X,A) 
	\end{align*}
	so that 
	\begin{align*}
		\dim \teich(f) \leq \card (B-A) - s(f) = \card S(f)  - p(f) - s(f),
	\end{align*}
	which is the desired inequality.
\end{proof}

\section{Proof of Theorem B}\label{sec:proofB}

In this section, we will focus on the case where $W=X=\rs$, so that $f: W \to X$ is 
a rational map. 
We will recover from the work done in the previous sections a simpler proof 
of a result due to Levin (see \cite{levin2011perturbations}). First let us introduce some notations.

If $\lambda \mapsto [\mu_\lambda]$ is a holomorphic curve in $\defab(f)$ passing through 
the basepoint $[0]$ at $\lambda=0$, then let  
$(\phi_\lambda,\psi_\lambda,f_\lambda)$ be a corresponding holomorphic family of triples
(recall that $\phi_\lambda$ and $\psi_\lambda$ are quasiconformal homeomorphisms, and that $f_\lambda$
are rational maps of the form $f_\lambda = \phi_\lambda \circ f \circ \psi_\lambda^{-1}$). 
Then
\begin{equation}\label{eq:diffdef}
	\eta = Df^{-1} \cdot \left(\frac{d}{d\lambda}_{|\lambda=0} f_\lambda \right) = f^*\dot \phi - \dot \psi,
\end{equation} 
and $\eta$ is a meromorphic vector field on $\rs$ that depends only on $\frac{d}{d\lambda}_{|\lambda=0} [\mu_\lambda] \in T_{[0]}\defab(f)$.

Recall the following notation:

\begin{defi}
	Let $v$ be a summable critical value and $\eta \in T(f)$. Denote by:
	$$\xi_\eta(v) := \sum_{k=0}^{\infty} (f^k)^*\eta(v) = \limn \eta_n(v) \in T_v \rs.$$
\end{defi}

We now come to our second result, that we state again here for the reader's convenience:

\begin{theoB}
	Let $f:\rs \to \rs$ be a rational map that is not a flexible Lattès map, with $s$ summable critical values $v_i$, $1 \leq i \leq s$, 
	such that for all $n \in \N^*$, $f^n(v_i) \notin V_f$. 
	Assume that either $f$ has no invariant line field, or that the $\omega$-limit set of those 
	$s$ summable critical values are $C$-compacts. Then there is a germ of analytic set $\Lambda$ in 
	$\ratd$ passing through $f$ and transverse to $\mathcal{O}(f)$ on which the critical values of $f$
	move holomorphically, and such that the linear map
	\begin{align*}
	\mathcal{V} : T(f) &\rightarrow \bigoplus_{1 \leq i \leq s} T_{v_i} \rs   \\
	\eta &\mapsto (\dot v_i + \xi_\eta(v_i))_{1 \leq i \leq s}
	\end{align*}
	has maximal rank, i.e. equal to $s$.
\end{theoB}

\begin{proof}
	Let $A$ be a repelling cycle of cardinal at least $3$ that does not 
	contain any critical value, and let $Z \subset A$ be a subset
	of cardinal exactly $3$. Let $B = A \cup S(f)$: then $\defab(f)$ is a complex manifold of 
	dimension $\card S(f)$.
	Let $\Phi: \defab(f) \to \modspace$ be the natural map from the deformation space to the moduli space,
	and let $\Phi_Z: \defab(f) \to \ratd$ be its lift to the parameter space obtained by 
	chosing quasiconformal homeomorphisms fixing $Z$ pointwise.
	Then $D\Phi_Z([0])$ takes values in $T_Z(f)$, the subspace of 
	$T(f)$ of vector fields vanishing on $Z$.
	For each summable critical value $v_i$, let $u_i$ be a non-zero element of $\ddf$ supported in
	$\{v_i\}$, and let $q_i$ be the quadratic differentials constructed in the proof of Theorem A
	(recall that $q_i$ is the $Z$-normalized quadratic differential
	 associated to $\sum_{n=0}^{\infty} f_*^n u_i$).
	
	If $\lambda \mapsto [\mu_\lambda]$ is a holomorphic curve in $\defab(f)$ tangent to $[\mu] \in T_{[0]}\defab(f)$
	at the basepoint, we can lift $\lambda \mapsto [\mu_\lambda]$ to a holomorphic curve of representatives
	$\lambda \mapsto \mu_\lambda$, and the normalized solutions $\psi_\lambda$ of the associated Beltrami
	equation
	will satisfy the property that
	$\dbar \dot \psi = \dbar \frac{d}{d\lambda}_{\lambda=0} \psi_\lambda$ is a representative of $[\mu]$.

	We will prove that
	that the linear map:
	\begin{align*}
		\mathcal{U} : T_{[0]}\defab(f) &\rightarrow \C^s   \\
		[\mu] &\mapsto (\langle u_i,\dot \psi(v_i)+\xi_\eta(v_i)\rangle)_{1 \leq i \leq s}
	\end{align*}
	has rank $s$.
	This will imply Theorem B. Indeed, since 
	$D\Phi_Z$ is injective, $\Phi_Z$ is a local biholomorphism onto its image, 
	and therefore defines a germ of analytic submanifold $\Lambda$ passing through $f$. It is 
	tangent to $\mathcal{O}(f)$, and the critical values of $f$ move holomorphically on
	$\Lambda$ since they are given by $v_i(\lambda)=\psi_\lambda(v_i)$.
	Additionnally, since the linear map
	\begin{align*}
		\bigoplus_{i \leq s} T_{v_i} \rs &\to \C^s\\
		(\xi_i)_{i \leq s} &\mapsto (\langle u_i, \xi_i \rangle )_{i \leq s}
	\end{align*}
	is an isomorphism, it is clear that if $\mathcal{U}$ has rank $s$ then $\mathcal{V}$ also
	has rank $s$.

	Let $\ucal_i : [\mu] \mapsto \langle u_i, \xi_\eta(v_i)\rangle$, so that $\ucal = (\ucal_i)_{i \leq s}$.
	Let $[\mu] \in T_{[0]}\defab(f)$ and $\eta = D\Phi_Z([0]) \cdot [\mu] \in T_Z(f)$.

	For $n \in \N$, denote by $q_{i,n}$ the $Z$-normalized quadratic differential associated to 
	$f_*^n u_i$: then $q_i = \sum_{n=0}^\infty q_{i,n}$ and $q_{i,n}$ is a quadratic differential 
	with exactly four poles, all simple, which are in $Z \cup \{y\}$.

	Note that
	$$2i\pi\sum_{n \in \N} \res(q_{i,n} \cdot \eta,f^n(v_i)) = \ucal_i([\mu]).$$
	
	Indeed, by definition:
	$$\langle f_*^n u_i, \eta \rangle = 2i\pi\res(q_{i,n} \cdot (f^n)^*\eta, f^n(v_i))$$
	and $\xi(v_i)=\sum_{n \geq 0} ((f^n)^*\eta)(v_i)$.

	The next lemma essentially says that the linear form $\ucal_i$ is represented by
	$\frac{1}{2i\pi}\nabla_f q_i$ in $T_{[0]}^*\defab(f)$:
	
	\begin{lem}\label{lem:calculsommable}
		Let $\mu = \dbar \dot \psi$ be a representative of $[\mu] \in T_0 \defab(f)$. Then :
		$$\int_{\rs} \nabla_f q_i \cdot \mu = \langle u_i, \dot \psi(v_i) + \xi_\eta(v_i)\rangle.$$
	\end{lem}

	\begin{proof}[Proof of lemma \ref{lem:calculsommable}]
		For all $n \in \N$, the differential form $q_{i,n} \cdot \eta$ is meromorphic on $\rs$,
		therefore the sum of its residues is null. The quadratic differential 
		$q_{i,n}$ has poles at $Z \cup \{f^n(v_i)\}$, and the vector field $\eta$ has poles at 
		$\critf$. So:
		$$\res(q_{i,n} \cdot \eta,f^{n}(v_i)) = - \sum_{z \in Z \cup \critf} \res(q_{i,n} \cdot \eta,z).$$
		
		Therefore:
		$$\sum_{n \in \N} \res(q_{i,n} \cdot \eta, f^n(v_i)) = - \sum_{n \in \N} \sum_{z  \in Z \cup C_f} \res(q_{i,n} \cdot \eta, z )$$
		
		According to equation $($\ref{eq:diffdef}$)$, we have: 
		$\eta = f^*\dot \psi - \dot \phi$, and $\dot \phi = \dot \psi$ on $Z \subset A$. 
		Therefore for all $z \in A$ and for all $n \in \N$ : 
		\begin{align*}
		\res(q_{i,n} \cdot \eta, z) &= \res(q_{i,n} \cdot (f^*\dot \psi - \dot \psi),z) = \res(f_* q_{i,n} \cdot \dot \psi,f(z))  - \res(q_{i,n} \cdot \dot \psi,z).
		\end{align*}
		Therefore:
		$$-\sum_{n \in \N} \sum_{z \in A} \res(q_{i,n} \cdot \eta, z) = \sum_{z \in A} \res(\nabla_f q_i \cdot \dot \psi,z).$$
		Moreover:
		\begin{align*}
		- \sum_{c \in \critf} \res(q_{i,n} \cdot \eta,c) &= - \sum_{c \in \critf} \res(q_{i,n} \cdot (f^*\dot \psi - \dot{\phi}),c)  \\
		&= - \sum_{c \in \critf} \res(q_{i,n} \cdot f^* \dot \psi,c) \\
		&= - \sum_{v \in V_f} \res(f_*q_{i,n} \cdot  \dot \psi,v).
		\end{align*}
		For all $n \geq 1$, since $q_{i,n}$ has no pole at any $v \in V_f$,
		we therefore have
		$$	- \sum_{c \in \critf} \res(q_{i,n} \cdot \eta,c) = \sum_{v \in V_f} \res(\nabla_f q_{i,n} \cdot \dot \psi,v).$$
		
		On the other hand, for $n=0$, $q_{i,0}$ has a pole at $v_i$, and so:
		$$	- \sum_{c \in \critf} \res(q_{i,0} \cdot \eta,c) = \sum_{v \in V_f} \res(\nabla_f q_{i,0} \cdot \dot \psi,v) - \res(q_{i,0} \cdot \dot \psi, v_i)$$
		
		Therefore:
		$$- \sum_{n \in \N} \sum_{c \in \critf} \res(q_{i,n} \cdot \eta,c) = - \res(q_{i,0} \cdot \dot \psi, v_i) + 
		 \sum_{v \in V_f} \res(\nabla_f q_{i} \cdot \dot \psi,v).$$
		To sum things up, we have:
		$$\sum_{n \in \N} \res(q_{i,n} \cdot \eta, f^n(v_i)) = -\res(q_{i,0} \cdot \dot \psi, v_i)+\sum_{z \in B} \res(\nabla_f q_i \cdot \dot \psi,z).$$
		Finally, by Stokes' theorem,
		$$2i\pi \sum_{z \in B} \res(\nabla_f q_i \cdot \dot \psi,z) = \int_{\rs} \nabla_f q_i \cdot \dbar \dot \psi,$$
		and by definition of $q_{i,0}$ (and the fact that $\dot \psi = 0$ on $Z$), 
		$$2i\pi\res(q_{i,0} \cdot \dot \psi, v_i) = \int_{\rs} q_{i,0} \cdot \dbar \dot \psi = \langle u_i, \dot \psi(v_i) \rangle.$$
		To conclude the proof of lemma \ref{lem:calculsommable}, just observe 
		that 
		$$2i\pi  \sum_{n \in \N} \res(q_{i,n}\cdot \eta, f^n(v_i)) = \sum_{n=0}^{\infty} \langle f_*^n u_i, \eta\rangle = \langle u_i, \xi_\eta(v_i)\rangle.$$
	\end{proof}
	
	Let us now return to the proof of Theorem B.
	According to the preceding lemma, we have:
	$$\ucal_i^Z([\mu]) = \frac{1}{2i\pi}\int_{\rs} \nabla_f q_i \cdot [\mu]$$ 
	for all $[\mu] \in T_{[0]}\defab(f)$. In other words, the class of the quadratic differential 
	$\frac{1}{2i\pi}\nabla_f q_i$ in $Q(\rs,B)/\nabla_f Q(\rs,A)$ represents the linear form $\ucal_i$
	in $T_{[0]}^*\defab(f)$.
	Moreover, according to lemma \ref{lem:qiindep}, the quadratic differentials $(\nabla_f q_i)_{i \leq s}$
	are linearly independent, and as in the proof of Theorem A, this implies 
	that the classes $([\nabla_f q_i])_{i \leq s}$ are linearly independent in 
	$Q(\rs,B)/\nabla_f Q(\rs,A)$.
	Therefore, the $(\ucal_i)_{i \leq s}$ are linearly independent, which proves that $\ucal$ has rank
	$s$.
\end{proof}

\bibliographystyle{alpha}
\bibliography{biblio}

\end{document}